\documentclass[10pt,a4paper]{article}

\usepackage{a4wide}
\usepackage{amsthm}
\usepackage{amsfonts}
\usepackage{amsmath}
\usepackage[english]{babel}

\usepackage{etex}
\usepackage{hyperref}
\usepackage{multimedia}
\usepackage[all]{xy}
\usepackage{amssymb}
\usepackage{graphicx,textpos}
\usepackage[dvipsnames]{xcolor}
\usepackage{helvet}
\usepackage{stmaryrd}
\usepackage{enumerate}
\usepackage{todonotes}
\usepackage{pdfsync}
\usepackage{dsfont}
\usepackage[shortcuts]{extdash}
\usepackage[all]{xy}
  \newdir{ >}{{}*!/-9pt/@{>}}

\newcommand{\N}{\mathbb{N}}
\newcommand{\R}{\mathbb{R}}
\newcommand{\Q}{\mathbb{Q}}
\newcommand{\Z}{\mathbb{Z}}
\newcommand{\C}{\mathbb{C}}

\newcommand{\I}{\mathds{1}}

\newcommand{\rst}[1]{\ensuremath{{\mathbin\mid}\raise-.5ex\hbox{$#1$}}}
\newcommand{\lie}{\mathfrak{g}}
\DeclareMathOperator{\Irr}{Irr}
\DeclareMathOperator{\Gal}{Gal}
\DeclareMathOperator{\GL}{GL}
\DeclareMathOperator{\Aut}{Aut}
\DeclareMathOperator{\Aff}{Aff}

\author{Karel Dekimpe and Jonas Der\'e\thanks{The second author was supported by a Ph.D.~fellowship of the Research Foundation -- Flanders (FWO).
Research supported by the research Fund of the KU Leuven}\\
KULeuven Kulak, E. Sabbelaan 53, BE-8500 Kortrijk, Belgium}
\title{\bf Existence of Anosov diffeomorphisms on infra-nilmanifolds modeled on free nilpotent Lie groups}
\date{\today}

\newtheorem{Def}{Definition}[section]

\newtheorem{Cor}[Def]{Corollary}
\newtheorem{Thm}[Def]{Theorem}
\newtheorem{Prop}[Def]{Proposition}
\newtheorem{Lem}[Def]{Lemma}
\newtheorem*{Rmk}{Remark}
\begin{document}

\maketitle

\begin{abstract}
An infra-nilmanifold is a manifold which is constructed as a quotient space $\Gamma\backslash G$ of a simply connected nilpotent Lie group $G$, where 
$\Gamma$ is a discrete group acting properly discontinuously and cocompactly on $G$ via so called affine maps. The manifold $\Gamma\backslash G$ 
is said to be modeled on the Lie group $G$. This class of manifolds is conjectured to be the only class of closed manifolds allowing an Anosov diffeomorphism.
However, it is far from obvious which of these infra--nilmanifolds actually do admit an Anosov diffeomorphism. In this paper we completely solve this 
question for infra-nilmanifolds modeled on a free $c$--step nilpotent Lie group. 
\end{abstract}
\section{Anosov diffeomorphisms on infra--nilmanifolds}

There has been quite some activity in the study of Anosov diffeomorphisms on closed manifolds in the last ten to fifteen years
(\cite{dm05-1,deki99-1,deki11-1, dd03-2, dv08-1, dv09-1, dv11-1, fg12-1, laur03-1,laur08-1,lw08-1, lw09-1, main06-1, main12-1, mw07-1,malf97-3, 
payn09-1}). Vaguely said, an Anosov diffeomorphism of a manifold $M$ is a diffeomorphism $f:M\to M$ such that the 
tangent bundle of $M$ allows a $df$--invariant continuous splitting $TM=E^s\oplus E^u$ with $df$ contracting on $E^s$ and 
expanding on $E^u$. Up till now, the only examples of closed manifolds admitting an Anosov diffeomorphism are manifolds which are 
homeomorphic to an infra--nilmanifold. In fact, it has been conjectured that this is the only class of closed manifolds admitting such 
diffeomorphisms. We refer to the papers cited above for more information.

\smallskip

Let us briefly recall what an infra--nilmanifold is and how an Anosov diffeomorphism can be constructed on such a manifold. 
Let $G$ be a connected and simply connected nilpotent Lie group and $\Aut(G)$ the group of continuous automorphisms of $G$. The affine group $\Aff(G)$ is defined as the semi-direct product $G \rtimes \Aut(G)$ and acts on $G$ in the following way:
\begin{eqnarray*}
\forall \alpha = (g, \delta) \in \Aff(G),\; \forall h \in G:\;\;  {}^\alpha h = g \delta(h).
\end{eqnarray*}
Let $C \subseteq \Aut(G)$ be a compact subgroup of automorphisms. A subgroup $\Gamma \subseteq G \rtimes C$ is called an almost--Bieberbach group if $\Gamma$ is a discrete, torsion-free subgroup such that the quotient $\Gamma \backslash G$ is compact. The quotient space $\Gamma \backslash G$ is a closed manifold and is called an infra-nilmanifold modeled on the Lie group $G$. Let $p: \Gamma \to \Aut(G)$ denote the natural projection on the second component, then it is well known that $H = p(\Gamma)$ is a finite group, which is called the holonomy group of $\Gamma$. The subgroup $N = \Gamma \cap G$ is a uniform lattice in $G$ and $\Gamma$ fits in the following exact sequence:
\begin{equation}\label{seq}
\xymatrix{1 \ar[r] & N \ar[r]  & \Gamma \ar[r] & H \ar[r] & 1}.
\end{equation}
In the case where $G$ is abelian, i.e.\ $G \cong \R^n$ for some $n$, the manifolds constructed in this way are exactly the compact flat Riemannian 
manifolds. 
 
\smallskip 

Let $\alpha \in \Aff(G)$ be an affine transformation such that $\alpha \Gamma \alpha^{-1} = \Gamma$. Then $\alpha$ induces a diffeomorphism $\bar{\alpha}$ on the infra--nilmanifold $\Gamma \backslash G$, which is defined by $$\bar{\alpha}: \Gamma \backslash G \to \Gamma \backslash G: \Gamma g \mapsto \Gamma ({}^\alpha g).$$ A diffeomorphism like this is called an affine infra-nilmanifold automorphism. The map $\bar{\alpha}$ is an Anosov diffeomorphism if and only if the linear part of $\alpha$ is hyperbolic, i.e.\ it only has eigenvalues of absolute value different from $1$. It is conjectured that every Anosov diffeomorphism is topologically conjugate to an affine infra-nilmanifold automorphism (see also \cite{deki11-1}).

\smallskip

In \cite{port72-1}, H.~L.~Porteous gives an easy criterion to decide whether  or not a flat manifold allows an Anosov diffeomorphism. 
To formulate this criterion, we have to introduce the holonomy representation. For a flat manifold, the short exact sequence \eqref{seq} is of the 
form $1 \to \Z^n \to \Gamma \to H \to 1$ and hence conjugation in $\Gamma$ induces  a representation $\varphi:H\to \Aut(\Z^n)=\GL_n(\Z)$.
This representation is called the holonomy representation. The criterion of Porteous states that a flat manifold $M$ supports 
an Anosov diffeomorphisms if and only if every $\Q$-irreducible component of $\varphi$ which occurs with multiplicity $1$, splits over $\R$. 
(We can also view $\varphi$ as a rational representation $\varphi:H \to \GL_n(\Q)$ and a real representation $\varphi: H \to \GL_n(\R)$.)
As it was already pointed out in \cite{dv08-1}, a natural class to consider for 
possibly generalizing this statement is the class of infra--nilmanifolds modeled on a free $c$--step nilpotent Lie group, 
i.e.\ Lie groups where the corresponding Lie algebra is free $c$--step nilpotent. 
The flat case is then the case where $c=1$.

\smallskip 

For infra--nilmanifolds which are not flat (so which are not modeled on an abelian Lie group), the short exact sequence \eqref{seq} does not give rise
to a natural holonomy representation $\varphi:H\rightarrow \Aut(N)$. Therefore, a rational holonomy representation was introduced in \cite{dv08-1}.
To obtain this rational holonomy representation, one embeds the group $N$ into its rational Mal'cev completion $N_\Q$ (or radicable hull).
This leads to the following commutative diagram:

$$ \xymatrix{
1 \ar[r] & N \ar[r] \ar@{ >->}[d]& \Gamma \ar[r] \ar@{ >->}[d] & H \ar[r] \ar@{=}[d]& 1 \\
1 \ar[r] & N_\Q \ar[r] & \Gamma_\Q \ar[r] & H \ar[r] & 1,}$$ where the bottom exact sequence splits. By fixing a splitting morphism $s: H \to \Gamma_\Q$, we define the rational holonomy representation $\varphi: H \to \Aut(N_\Q)$ by $$\varphi(f)(n) = s(f) n s(f)^{-1}.$$ 
The rational holonomy representation does depend on the choice of $s$, but this dependence is not relevant in the study of Anosov diffeomorphisms. In fact, it turns out that the rational holonomy representation contains all information about the existence of Anosov diffeomorphisms, see e.g.\ \cite[Theorem A]{dv08-1}. The representation $\varphi$ also induces a representation 
$$\bar{\varphi}: H \to \Aut(N_\Q / [N_\Q,N_\Q]) \cong \GL_n(\Q)\ \  \mbox{(for some $n$)},$$ 
which will be refered to as the abelianized rational holonomy representation. 
Our main theorem, which generalizes the criterion of Porteous and completely finishes the work started in \cite{dv08-1},  states that we can decide whether or not an infra-nilmanifold modeled on a free 
$c$--step nilpotent Lie group has an Anosov diffeomorphism by only looking at the abelianized rational holonomy representation $\bar{\varphi}$:

\begin{Thm}\label{main}
Let $M$ be an infra-nilmanifold modeled on a free $c$--step nilpotent Lie group, with holonomy group $H$ and associated abelianized rational holonomy representation $\overline{\varphi}: H \to \Aut \left( \frac{N_\Q}{[N_\Q,N_\Q]} \right)$. Then the following are equivalent: 
\begin{center}
$M$ admits an Anosov diffeomorphism. \\
$\Updownarrow$ \\
Every $\Q$-irreducible component of
$\bar{\varphi}$ that occurs with multiplicity $m$,\\ splits in more
than $\frac{c}{m}$ components when seen as a representation over
$\R$.
\end{center}
\end{Thm}

For abelian holonomy groups $H$, this was already shown in \cite[Theorem 2.3]{dv11-1}.
From \cite{dv11-1}, we know that the existence of an Anosov diffeomorphisms is equivalent with the existence of a $c$-hyperbolic, integer-like matrix which commutes with every element in the image of $\overline{\varphi}$. Recall that a matrix $C\in \GL_n(\Q)$
is called integer-like if its characteristic polynomial has coefficients in $\Z$ and its determinant is $\pm 1$. A matrix is called $c$-hyperbolic if there does not exists a natural number $k \leq c$ such that the product of any $k$ (not necessarily different) eigenvalues 
of the matrix has absolute value equal to $1$. By decomposing the representation $\overline{\varphi}$ into its $\Q$-irreducible components, it was shown in \cite{dv11-1} that the following theorem implies our main theorem:

\begin{Thm} \label{vert}
Let $H$ be a finite group and $\rho: H \to \GL_n(\Q)$ a $\Q$-irreducible representation. Then there exists a $c$-hyperbolic, integer-like matrix $C\in \GL_{mn}(\Q)$ 
which commutes with 
$m \rho=\underbrace{\rho\oplus \rho \oplus \cdots \oplus \rho}_{m\ {\rm times}}$ 
if and only if $\rho$ splits in strictly more than $\frac{c}{m}$ components when seen as a representation over $\R$.
\end{Thm}

In the second section we show the first direction of our main theorem by just looking at the dimension of the real components of a $\Q$-irreducible representation. The other direction is harder and we will need some preliminary results about number theory and splitting fields of representations. The work for this direction is done in the third section. In the last section we shortly explain how the results 
of this paper can be generalized to some more classes of infra-nilmanifolds.  Throughout this paper, we will use results from representation theory of finite groups. We will give explicit references at several places. When needed, the reader can consult  \cite{isaa76-1} (mainly chapters 9 and 10) and 
\cite{serr77-1} for more information.

\section{A condition on the number of components}\label{firstdirection}
All fields we will consider in this paper are assumed to be subfields of the field $\C$ of complex numbers.
Let $G$ be a finite group and $\Irr(G)$ the set of $\C$-irreducible characters of $G$. 
Let $F$ be a field and $\chi$ a character (not necessarily irreducible), then we say that $\chi$ is afforded by a $F$-representation if there exists a representation $\rho: G \to \GL_n(F)$ such that the character of $\rho$ equals $\chi$. For each $\chi \in \Irr(G)$, we can define the complex number $$\nu_2(\chi) = \frac{1}{|G|} \sum_{g \in G} \chi(g^2).$$ It is well known (see e.g. \cite[page 58]{isaa76-1} or \cite[Prop.\ 39, page 109]{serr77-1}) 
that $\nu_2(\chi)$ can only have values $1,0$ and $-1$ with the following possibilities:
\begin{itemize}
\item $\nu_2(\chi) = 0$ if and only if $\chi$ is not real (i.e.\ $\exists g \in G$ such that $\chi(g) \notin \R$).
\item $\nu_2(\chi) = 1$ if and only if $\chi$ is afforded by a real representation.
\item $\nu_2(\chi) = -1$ if and only if $\chi$ is real but not afforded by a real representation.
\end{itemize}

If $\chi$ is an irreducible character, we write $\Q(\chi)$ for the smallest subfield of $\C$ that contains $\chi(g)$ for all $g \in G$. The field $K = \Q(\chi)$ is Galois over $\Q$ and for every $\sigma \in \Gal(K,\Q)$, we can define the map $$\chi^\sigma: G \to K: g \mapsto \sigma(\chi(g)),$$ which we call a Galois conjugate of $\chi$. Every Galois conjugate of $\chi$ is also an irreducible character of $G$ (see \cite[Lemma 9.16]{isaa76-1}). 
It is easy to show that the invariant $\nu_2$ is the same for all of these Galois conjugates:

\begin{Lem}
\label{nu}
Let $\chi \in \Irr(G)$ and $\sigma \in \Gal\left(\Q(\chi),\Q\right)$, then we have that $\nu_2(\chi) = \nu_2(\chi^\sigma).$
\end{Lem}

\begin{proof} By using the formula for $\nu_2$, we have 
\begin{eqnarray*}
\nu_2(\chi^\sigma) = \frac{1}{|G|} \sum_{g \in G} \chi^\sigma(g^2) &=& \frac{1}{|G|} \sum_{g \in G} \sigma( \chi(g^2)) \\
& = & \sigma(\frac{1}{|G|} \sum_{g \in G} \chi(g^2)) \\ & = & \sigma ( \nu_2(\chi)) = \nu_2(\chi)
\end{eqnarray*}
since $\nu_2(\chi) \in \Z$.
\end{proof}

This easy lemma has the following important corollary about the dimension of irreducible components:

\begin{Cor}
{All $\R$-irreducible components of a $\Q$-irreducible representation have the same dimension.}
\end{Cor}
\begin{proof} Let $\chi$ be the character of a $\C$-irreducible component of the $\Q$-irreducible representation $\rho$. 

If $\nu_2(\chi) = 1$, then all irreducible components of $\rho$ are afforded by a real representation, since they are all of the form $\chi^\sigma$ for some $\sigma \in \Gal\left( \Q(\chi), \Q \right)$ (\cite[Lemma 9.21 (c)]{isaa76-1}). Thus the $\R$-irreducible components of $\rho$ are equal to the $\C$-irreducible components and they all have the same dimension $\chi(e_G)$. 

If $\nu_2(\chi) \neq 1$, then none of the irreducible components is afforded by a real representation and thus every $\R$-irreducible component has the same dimension $\chi(e_G) + \overline{\chi}(e_G) = 2  \chi(e_G)$.  \end{proof}

\smallskip

The next proposition follows naturally from the corollary:

\begin{Prop}\label{det}
Let $\rho: G \to GL_{kr}(\Q)$ be a $\Q$-irreducible representation with $r$ the dimension of every $\R$-irreducible component of $\rho$. If $C \in GL_{krm}(\Q)$ commutes with $m \rho = \rho \oplus \rho \oplus \ldots \oplus \rho$ and $|\det(C)| = 1$, then $C$ cannot be $km$-hyperbolic.
\end{Prop}
\begin{proof} By using the real Jordan canonical form of $C$ (see e.g. \cite[Theorem 3.4.5.]{hj96-1}), we find an invertible matrix $A \in GL_{krm}(\R)$ such that $$ A C A^{-1} = \begin{pmatrix} C_1 & 0 & \hdots & 0 \\ 0 & C_2 & \hdots & 0  \\ \vdots & \vdots& \ddots & \vdots\\0 &0 & \hdots  & C_l	\end{pmatrix}$$ where $C_i$ and $C_j$ have distinct eigenvalues for $i \neq j$ and such that $C_i$ either has one real eigenvalue or two complex conjugate eigenvalues. Let $\lambda_i$ be one (of the at most two) eigenvalue(s) of $C_i$. So all $\lambda_i$ are distinct by construction. By conjugating with the same matrix $A$, the representation $m \rho$ also splits over $\R$, since the generalized eigenspaces of $C$ are obviously invariant under the representation $m\rho$. This implies that each $C_i$ has dimension $k_i r$ for some $k_i \in \N_0$, because all real components of $m \rho$ have the same dimension $r$. So for the determinant of $C$, we have
\begin{eqnarray*}
1 = |\det(C)| = \prod_{i=1}^l |\det(C_i)| &=& \prod_{i=1}^l |\lambda_i|^{k_ir}  \\ &=&  \left( \prod_{i=1}^l |\lambda_i|^{k_i} \right)^r 
\end{eqnarray*}
and therefore $\prod_{i=1}^l |\lambda_i|^{k_i} = 1$. This means that $C$ is not $km$ hyperbolic, since $\sum_{i=1}^l k_i  = km$ by looking at the size of $C$. \end{proof}

Proposition~\ref{det} already implies one direction of Theorem~\ref{vert}. Indeed, if there exists a $c$-hyperbolic, integer-like matrix that commutes with $m \rho$, then the $k$ of Proposition~\ref{det} must be strictly larger than $\frac{c}{m}$. Note that this $k$ is the number of $\R$-irreducible components of $\rho$ and thus gives us one direction of Theorem~\ref{vert}.

\section{Construction of an Anosov diffeomorphism}

The proof of the other direction needs some more work and uses different results of number theory and representation theory for finite groups. The first step is investigating how a $\Q$-irreducible representation splits over certain fields, especially over minimal splitting fields. Next, we consider the existence of $c$-hyperbolic units in number fields, which correspond to $c$-hyperbolic matrices. By combining these results, we can construct a $c$-hyperbolic, integer-like matrix $C$ which commutes with the given representation. In this section, we assume that all groups $G$ are finite.

\subsection{Decomposition over a minimal splitting field}
If $\rho:A\rightarrow \GL_n(\Q)$ is a $\Q$-irreducible representation of a finite abelian group $A$, then its image is always cyclic. 
Indeed, Schur's lemma states that the set of $A$--endomorphisms of the rational vectorspace $\Q^n$ is a skew field. Note that $\rho(A)$ is an abelian subgroup 
of the multiplicative group of this skew field, hence $\rho(A)$ is cyclic (see e.g.\ \cite{hers53-1}). It follows also that a
generator $\rho(a)$ of $\rho(A)$ has a cyclotomic polynomial as its characteristic polynomial (see \cite[pages 570--571]{dv08-1} for more details). 

\medskip

A consequence is that in this case every $\Q$-irreducible representation has a basis of eigenvectors which are Galois conjugates. In this section we answer the question how a general representation splits over a minimal splitting field. For the formulation of this result we will make use of permutation matrices.

\medskip

Let $K=\Q(\theta) \supseteq \Q$ be an extension of degree $n$ and $\{\sigma_1, \ldots, \sigma_n\}$ the set of monomorphisms $K \to \C$. We say that $\sigma_i$ is real if $\sigma_i(K) \subseteq \R$, otherwise we call $\sigma_i$ complex. If $\sigma_i$ is complex then $\overline{\sigma_i}$ is a different monomorphism $K\to \C$, so the complex monomorphisms come in pairs. We conclude that $n = s + 2t$ with $s$ the number of real monomorphisms and $2t$ the number of complex monomorphisms. Let us remark that every monomorphism $\sigma_i$ is completely determined by the image of $\theta$, namely $\theta_i = \sigma_i(\theta)$. If $t = 0$, so if every $\sigma_i$ has a real image, then we call the field $K$ totally real. When $s=0$ we will say that the field is totally imaginary.

\smallskip

Let $\rho$ be a $\Q$-irreducible representation of a finite group $G$. From \cite[Theorem 9.21]{isaa76-1} we know that each $\C$-irreducible component of $\rho$ occurs with the same multiplicity $m$. If $\chi \in \Irr(G)$ is the character of one of those components, then the set of characters of all $\C$-irreducible components of $\rho$ is given by the Galois conjugates of $\chi$. So if we have $\{\sigma_1, \ldots, \sigma_n\} = \Gal(\Q(\chi),\Q)$, we can write the character $\chi_\rho$ of $\rho$ as 
$$\chi_\rho = m \chi^{\sigma_1} + \ldots + m \chi^{\sigma_n}.$$ 
The fixed multiplicity $m$ is called the Schur Index of $\chi$ over $\Q$ and is written as $m_\Q(\chi)$ (see \cite[pages 160--161]{isaa76-1}). The number 
$m_\Q(\chi)$ is also the smallest integer $m$ such that $m\chi$ is afforded by a $\Q(\chi)$--representation. Moreover, it is also given by the minimal degree of a field extension $\Q(\chi) \subseteq F$ such that $\chi$ is afforded by an $F$-representation (\cite[Theorem 10.17]{isaa76-1}). 
Such a field $F$ is called a minimal splitting field for the character $\chi$. Note that the number of $\C$-irreducible components of $\rho$ is equal to $m\cdot n = [F:\Q(\chi)]\cdot [\Q(\chi):\Q]=[F:\Q]$.

\smallskip

Let $\pi \in S_n$ be a permutation, then there exists a permutation matrix $K_\pi \in \GL_n(\Z)$, which is defined by $$(K_\pi)_{ij} = \left\{ \begin{matrix}1 & j = \pi(i) \\ 0 & \mbox{otherwise} \end{matrix}\right.$$ Note that the relation $K_{\pi_1} K_{\pi_2} = K_{\pi_2 \pi_1}$ holds for all $\pi_1, \pi_2 \in S_n$, so in particular we have that $K_{\pi}^{-1} = K_{\pi^{-1}}$. 
Let $M$ be a matrix in $GL_n(\C)$, then $K_\pi M$ is the matrix $M$ where the rows are permuted according to $\pi$, so the $i$-th row of $K_\pi M$ is the $\pi(i)$-th row of $ M$.
If we write $$M = \begin{pmatrix}
m_1 \\ m_2 \\ \vdots \\ m_n 
\end{pmatrix},$$ as a column of row vectors, then another way of saying this is that $$K_\pi M = \begin{pmatrix}
m_{\pi(1)} \\ m_{\pi(2)} \\ \vdots \\ m_{\pi(n)} 
\end{pmatrix}.$$
In the same way, $M K_\pi$ is the matrix $M$ where the columns are permuted according to $\pi^{-1}$. This last property can easily be checked by observing that the transpose $(K_\pi)^T=K_{\pi^{-1}}$ and then writing 
$M K_\pi=(K_{\pi^{-1}} M^T) ^T$. The Kronecker product $K_\pi \otimes \I_k \in \GL_{kn}(\Z)$ will be denoted as $K_\pi^{\otimes k}$. If we write a matrix $M \in GL_{kn}(\C)$ as a column of $k \times kn$ matrices $$M = \begin{pmatrix}
m_1 \\ m_2 \\ \vdots \\ m_n 
\end{pmatrix},$$ then similarly as above $K_\pi M$ is the matrix $M$ with the block matrices $m_i$ permuted according to the permutation $\pi$.

\smallskip

Let $E$ be a finite Galois extension over $\Q$. Take a subfield $F = \Q(\theta) \subseteq E$, which is not necessarily Galois over $\Q$, and denote by $\sigma_1,\ldots,\sigma_n$ the $n$ distinct monomorphisms $F \to \C$. Because $\theta_i = \sigma_i(\theta) \in E$, we have that $\sigma_i(F)\subseteq E$ for all $i$. 
It follows that for any $\sigma \in \Gal(E,\Q)$ and any $i\in \{1,2,\ldots,n\}$, 
we have that $\sigma \circ \sigma_i:F \to E \subseteq \C$ is again a monomorphism of 
$F$ in $\C$. Hence $\sigma$ induces a permutation on $\{\sigma_1,\sigma_2,\ldots,
\sigma_n\}$. Associated to this is a permutation $\pi_\sigma \in S_n$ which is determined by 
\[ \sigma \circ \sigma_i = \sigma_{\pi_\sigma(i)}.\]
We let $K_\sigma \in \GL_n(\Z)$ denote the corresponding
permutation matrix.

\smallskip

We already mentioned above that for every $\Q$-irreducible representation $\rho$ of a finite abelian group, one can always find a basis of eigenvectors which are Galois-conjugates, see \cite{dv11-1}. This means that if $E$ is a minimal splitting field of $\rho$ with $\Gal(E,\Q) = \{ \sigma_1, \ldots, \sigma_n \}$, there exists a vector $v_0 \in E^n$ such that with respect to the basis 
$\{ v_1=\sigma_1(v_0), v_2=\sigma_2(v_0), \ldots, v_n = \sigma_n(v_0)\}$, the representation is diagonal. 

\smallskip

The main part of this subsection is to generalize this statement to arbitrary groups:
\begin{Thm}
\label{sigmas}
Let $\chi \in \Irr(G)$ be the character of an irreducible component of a $\Q$-irreducible representation $\rho$. Let $\Q \subseteq F$ be a field extension of minimal degree such that $\chi$ is afforded by a $F$-representation, say $\rho_0$ and assume that $k$ is the dimension of this representation. If $[F:\Q] = n$ and $\sigma_1, \ldots, \sigma_n$ are the $n$ distinct monomorphisms $F \to \C$ and if $E$ is any field extension $\Q \subseteq F \subseteq E$ such that $E$ is Galois over $\Q$, then there exist a $P \in \GL_{kn}(E)$ such that the following conditions hold:
\begin{enumerate}[(1)]
\item  $P^{-1} \rho P = \begin{pmatrix}
\sigma_1(\rho_0) & 0 & \hdots & 0\\
0 & \sigma_2(\rho_0) & \hdots & 0 \\
\vdots & \vdots & \ddots & \vdots \\
0 & 0 & \hdots & \sigma_n(\rho_0)
\end{pmatrix}.$ 
\item  For all $\sigma \in \Gal(E,\Q)$: \  $\sigma(P)= P K_{\sigma^{-1}}^{\otimes k}$.
\end{enumerate}
\end{Thm}

\begin{proof}
Let $\rho_0: G \to GL_k(F)$ be an irreducible representation with character $\chi$. 
Take the representations $\rho^i = \sigma_i(\rho_0): G \to \GL_k(\sigma_i(F))$ with character $\chi^{\sigma_i}$. We construct a new representation 
$$\tilde{\rho}: G \to \GL_{kn}(E): g \mapsto \begin{pmatrix}
\rho^1_g & 0 & \hdots & 0\\
0 & \rho^2_g & \hdots & 0 \\
\vdots & \vdots & \ddots & \vdots \\
0 & 0 & \hdots & \rho^n_g
\end{pmatrix} ,$$ which was already used in the formulation of the theorem. 

\smallskip

Since $F$ is a finite field extension of $\Q$, we know that $F = \Q(\theta)$ for some algebraic number $\theta$. Look at the matrix $$ Q = \begin{pmatrix}
\sigma_1(1) & \sigma_1(\theta) & \hdots & \sigma_1(\theta^{n-1}) \\
\sigma_2(1) & \sigma_2(\theta) & \hdots & \sigma_2(\theta^{n-1}) \\
\vdots & \vdots & \ddots & \vdots \\
\sigma_n(1) & \sigma_n(\theta) & \hdots & \sigma_n(\theta^{n-1})
\end{pmatrix} \otimes \I_k,$$ then it follows by construction that $\sigma(Q) = K^{\otimes k}_\sigma Q$ for every $\sigma \in \Gal(E,\Q)$. Note that $Q$ is the Kronecker product of a Vandermonde matrix with the identity and thus its determinant can easily be computed and is different from 0. So $Q$ is invertible and denote the inverse of $Q$ by $P$. By applying $\sigma \in \Gal(E,\Q)$ to the relation $Q P = \I_{kn}$, 
we find that $$\sigma(P) = \sigma(Q)^{-1} = Q^{-1} (K_\sigma^{\otimes k})^{-1} = P K^{\otimes k}_{\sigma^{-1}}$$ and thus the matrix $P$ satisfies condition $2$ of the theorem. 

\smallskip

First we show that $P \tilde{\rho}_g Q \in \GL_{nk}(\Q)$ for all $g \in G$. Equivalently, we have to show that $\sigma(P \tilde{\rho}_g Q) = P \tilde{\rho}_g Q$ for all $\sigma \in \Gal(E,\Q)$. We compute that 
\begin{eqnarray*}
\sigma(P \tilde{\rho}_g Q) = \sigma(P) \sigma(\tilde{\rho}_g) \sigma (Q) = P K_{\sigma^{-1}}^{\otimes k} \sigma(\tilde{\rho}_g) K_{\sigma} ^{\otimes k} Q.
\end{eqnarray*} 
It is easy to see that $\sigma(\tilde{\rho}_g)= K_{\sigma} ^{\otimes k} \tilde{\rho}_g K_{\sigma^{-1}}^{\otimes k}$ and thus the conclusion follows. So it suffices to show that the representations $P \tilde{\rho} Q$ and $\rho$ are equivalent over $\Q$. 

\smallskip

Note that $\tilde{\rho}$ and thus also $P \tilde{\rho}Q$ has character $\chi^{\sigma_1} + \ldots + \chi^{\sigma_n}$. So the representations $P \tilde{\rho} Q$ and $\rho$ have a common $\C$-irreducible factor, namely $\rho_0$, corresponding to the character $\chi$. Also, it follows from the discussion above about the characters of $\Q$-irreducible representations that both have the same dimension, namely $kn$. As a consequence of \cite[Corollary 9.7]{isaa76-1}, we have that $P \tilde{\rho}Q$ has a $\Q$-irreducible component which is equivalent with $\rho$ over $\Q$. Because they have the same dimension, this ends the proof of the theorem.
\end{proof}

It's easy to see that this theorem is in fact a generalization of the statement for abelian groups.
The advantage of working with a matrix $P$ that satisfies the conditions of the theorem is that we can easily construct matrices that have coefficients in $\Q$:
\begin{Prop}
\label{q}
Let $F \supseteq \Q$ be a field extension of degree $n$ and $\sigma_1, \ldots, \sigma_n$ the distinct monomorphisms from $F$ to $\C$. Let $C_0$ be any matrix with coefficients in $F$ and look at the matrix $$C = \left(\begin{array}{cccc} \sigma_1(C_0) & 0 & \cdots & 0 \\
0 & \sigma_2(C_0) & \cdots & 0 \\
\vdots & \vdots & \ddots & \vdots \\
0 & 0 & \cdots & \sigma_n(C_0)\end{array}\right).$$ Let $E$ be any field extension of $F$ which is Galois over $\Q$ and $P$ an invertible matrix with entries in $E$ which satisfies $$\sigma(P) = P K^{\otimes k}_{\sigma^{-1}}$$ for all $\sigma \in \Gal(E,\Q)$, where we use notations as above. Then $ P C P^{-1}$ is a matrix with coefficients in $\Q$.
\end{Prop}

The proof of this proposition is immediate. Note that the matrix $C$ of the previous proposition doesn't depend on the choice of Galois extension $E$. Also the construction of the matrix $P$ in the proof of Theorem \ref{sigmas} didn't depend on the field $E$. It was only used to embed the monomorphisms $\sigma_i$ in a nice group structure. In the rest of this article we will ignore the use of the Galois extension $E$.

\subsection{Existence of real minimal splitting fields}

For a $\Q$-irreducible representation $\rho$ of a finite cyclic group, there is always a canonical choice of a minimal field extension $F \supseteq \Q$ such that the representation is completely reducible over $F$, i.e.\ is diagonalizable over $F$. If $f \in \Q[X]$ is the characteristic polynomial of a generator of the image of $\rho$, then the field $F$ is equal to the splitting field of $f$ over $\Q$. For non-abelian groups, we explained above that such a minimal splitting field $F$ also exists, but it is far from unique in general. In this subsection we show how to construct a real minimal splitting field in the case where $\rho$ is completely reducible over $\R$.

\medskip

An example of the non-uniqueness of the minimal splitting field can be given by the unique two-dimensional irreducible representation of the quaternion group $Q_8$. Take complex numbers $\alpha,\beta \in \C$ with $\alpha^2 + \beta^2 = -1$ and look at the representation $\rho$ given by 
 \begin{align*}
\rho(-1) & = \begin{pmatrix} -1 & 0 \\0 & -1 \end{pmatrix}\\
\rho(i) & = \begin{pmatrix} 0 & - 1 \\  1 & 0 \end{pmatrix}\\
\rho(j) &=   \begin{pmatrix} \alpha & \beta \\  \beta & -\alpha \end{pmatrix}.
 \end{align*}
It's an easy exercise to check that this indeed defines a representation of $Q_8$. This shows that every field of the form $\Q(\sqrt{-1-\alpha^2})$ with $\alpha \in \Q$ is a minimal splitting field for this representation. These fields do have in common that they are not real. In fact it's an exercise to show that every minimal splitting field for this character is totally imaginary.

\smallskip

In general it's not true that if a character $\chi$ is afforded by an $E$-representation for a field $E \supseteq \Q$, that there is also a subfield $F \subseteq E$ such that $F$ has minimal degree and $\chi$ is afforded by a $F$-representation. Therefore we cannot directly conclude that every real representation has a real minimal splitting field. If $m_\Q(\chi) = 1$, or said differently, if $\Q(\chi)$ is a splitting field for $\rho$, then this must of course be true. 
By the Brauer-Speiser Theorem, see for example \cite[page 171]{isaa76-1}, we know that $m_\Q(\chi) \leq 2$ and thus the only situation left to check is the one with $m_\Q(\chi) = 2$.

\smallskip

The following lemma characterizes the existence of a real minimal splitting field:

\begin{Lem}
\label{square}
Let $\chi$ be an irreducible real valued character of $G$ with $m_\Q(\chi)=2$. Let $K = \Q(\chi) \subseteq \R$ and $\rho: G \to \GL_{2n}(K)$ a representation with character $2 \chi$. Then the following statements are equivalent:
\begin{enumerate}[(1)]
\item There exists a minimal splitting field $F \subseteq \R$.
\item There exists a $G$-isomorphism $f: K^{2n} \to K^{2n}$ with $f^2 = \kappa \I_{K^{2n}}$, $\kappa > 0$ and $\sqrt{\kappa} \notin K$. 
\end{enumerate}
\end{Lem}

\begin{proof}
First we show that $(1)$ implies $(2)$. Let $F$ be such a minimal splitting field, so $F = K(\sqrt{d})$ with $d \notin K^2$, $d>0$ and there exists an $F$-representation $\rho_0$ which affords the character $\chi$. Denote by $\sigma$ the nontrivial element of $\Gal(F,K)$. By using the same techniques as in the proof of Theorem \ref{sigmas}, we can show that there exists a matrix $P$ such that
$$P^{-1} \rho P = \begin{pmatrix} \rho_0 & 0 \\ 0 & \sigma(\rho_0) \end{pmatrix},$$ with $\sigma(P) = P K^{\otimes n}_\sigma$. Now consider the matrix 
$$C=\begin{pmatrix} \sqrt d \I_n & 0 \\ 0 & \sigma(\sqrt d) \I_n \end{pmatrix},$$ then it follows, just like in Proposition \ref{q}, that $P C P^{-1}$ has coefficients in $K$ and commutes with the representation $\rho$. It's easy to check that the linear map $f$ induced by this matrix satisfies all the wanted properties.

\smallskip

Next we prove that the existence of $f$ gives us a minimal splitting field. Take $F = K(\sqrt{\kappa})$ and $\sigma$ as before. Consider the representation $\rho$ as an $F$-representation by extending the scalars, in the same way the map $f$ is also an $G$-isomorphism over $F$. Since $f^2 = \kappa \I_{F^{2n}}$, we know that $f$ can only have two different eigenvalues, namely $\sqrt{\kappa}$ and $- \sqrt{\kappa}$. Each of those eigenvalues occurs, since if $v$ is an eigenvector for $\sqrt{\kappa}$, then $\sigma(v)$ is an eigenvector for $\sigma(\sqrt{\kappa}) = - \sqrt{\kappa}$ and vice versa. Now the splitting of $F^{2n}$ into the eigenspaces of $f$ gives us a decomposition of $F^{2n}$ into two $G$-invariant subspaces and thus $F$ is a splitting field. Because $F \subseteq \R$, this ends our proof.
\end{proof}

We use this lemma to prove the existence of a real minimal splitting field:

\begin{Thm}\label{nietsterk}
Let $\chi \in \Irr(G)$ be an irreducible character which is afforded by a real representation. Then there exists a real minimal splitting field for $\chi$. 
\end{Thm}

\begin{proof} We use some ideas of the proof of \cite[Theorem 31]{serr77-1}, but in our case, we work over a finite field extension of $\Q$ instead of over $\C$.
Let $K = \Q(\chi)$, then we know that $K$ is a real field extension of $\Q$. As mentioned above, we only have to check the case where $m_\Q(\chi)=2$ because of the Brauer-Speiser Theorem. So there exists a field $F \supseteq K$ such that $\chi$ is afforded by a $F$-representation $\rho$ and $[F:K]=2$. If $F \subseteq \R$, there is nothing to prove, so we can assume that $F = K(\theta)$ with $\theta = \sqrt{d}, \hspace{1mm} d \in K$ and $d <0$.

\smallskip

Recall that $\rho$ also induces a representation on the dual vector space $V^\ast$, namely $\rho^\ast: G \to GL(V^\ast)$ with $$\rho^\ast_g(\varphi) = \varphi \circ (\rho_g)^{-1} = \varphi \circ \rho_{g^{-1}} \hspace{5mm} \forall \varphi \in V^\ast.$$ Since the values of $\chi$ lie in $\R$, the character of $\rho^\ast$ is equal to $\chi$. So the representations $\rho$ and $\rho^\ast$ are equivalent and there exists a $G$-isomorphism $f: V \to V^\ast$. 
Now look at the map $$B: V \times V \to F: (v,w) \mapsto f(v)(w).$$ It is easy to see that $B$ is a nondegenerate bilinear form on $V$. Since $f$ is a $G$-isomorphism, $B$ is clearly $G$-invariant, i.e.\ $B(\rho_g(v),\rho_g(w)) = B(v,w)$ for all $g \in G, \hspace{1mm} v,w \in V$. This construction gives us an isomorphism between the space of $G$-morphisms from $V$ to $V^\ast$ and the space of $G$-invariant bilinear forms on $V$. Since $V$ and $V^\ast$ are irreducible, the space of $G$-morphisms between them has dimension $1$ . Hence, this $G$-invariant isomorphism $f$ is unique up to scalar multiplication.

\smallskip

We can also consider the spaces $V_\C=\C \otimes_{F} V$ and $(V_\C)^\ast=(V^\ast)_\C= \C \otimes_{F} V^\ast $ and the induced representations $\rho_\C:G\to \GL(V_\C)$ and 
$\rho^\ast_\C:G\to \GL(V^\ast_\C)$.  Then $f$ extends to a $G$-isomorphism $f_\C: V_\C \to V^\ast_\C$ and the bilinear form $B$ extends to $B_\C:V_\C\times V_\C\to \C: (v,w)\mapsto f_\C(v)(w)$.
Since $\chi$ is afforded by a $\R$-representation, there exists also a nondegenerate symmetric bilinear form over $\C$ which is invariant under $G$ by \cite[Theorem 31]{serr77-1}. Because of the uniqueness of nondegenerate $G$-invariant bilinear forms up to scalar multiplication, it now follows that $B_\C$, and hence also $B$, must be symmetric.

\smallskip

Now choose a $G$-invariant, positive definite, hermitian scalar product 
\[ \left\langle \ ,\ \right \rangle:V \times V \rightarrow F.\]
For every $v \in V$, there exists a unique $\psi(v)\in V$ such that $B(w,v) = \left\langle w,\psi(v)\right\rangle$ for all $w \in V$. An easy computation then shows that $\psi$ is bijective and antilinear, i.e.\ $\psi(\lambda v) = \overline{\lambda} \psi(v)$ for all $\lambda \in F$. Also, $\psi$ is $G$-invariant because both $B$ and $\langle \hspace{1mm},\hspace{1mm} \rangle$ are $G$-invariant. So $\psi^2$ is a $G$-automorphism of $V$ and thus $\psi^2 = \mu \I_V$ for some $\mu \in F$. By looking at $V$ as a vector space over $K$, the map $\psi$ becomes linear (since $K \subseteq \R$). Therefore, it is sufficient to show that $\mu \in K, \mu > 0$ and $\sqrt{\mu} \notin K$ because of Lemma \ref{square}.

\smallskip

For any $v,w\in V$ we have that 
\[ \langle v, \psi(w)\rangle = B(v,w) =B(w,v)= \langle w, \psi(v)\rangle =
\overline{\langle \psi(v), w\rangle}.\]
Using this identity, we find for all $v,w \in V$ that 
\[
\overline{\mu} \langle v, w \rangle = \langle v, \mu w\rangle= \langle v, \psi^2(w) \rangle =
\overline{\langle\psi(v),\psi(w)\rangle}= \langle\psi^2(v) , w \rangle=\mu\langle v,w\rangle
\]
and thus $\mu \in \R \cap F = K$. From the same computation with $v=w$, we also have that $\mu \langle v,v\rangle = \langle \psi(v), \psi(v) \rangle$ and $\mu >0$ because $\langle  \ ,\ \rangle$ is positive definite. It's easy to check that $\sqrt{\mu} \notin K$ because $\chi$ cannot be afforded by a $K$-representation. This ends the proof.
\end{proof}

\subsection{Existence of c-hyperbolic units in number fields}

In the previous parts we discussed the minimal splitting field for a representation. It is now important to know if we can find units in these field which are sufficiently hyperbolic.

\smallskip

Let $K = \Q(\theta)$ be any number field of degree $n$ with monomorphisms $\sigma_i: K \to \C$ for $i \in \{1, \ldots, n\}$. Then we denote by $\mathcal{O}_K \subseteq K$ the subring of algebraic integers and $U_K$ the group of units in $\mathcal{O}_K$. Note that the elements of $U_K$ are exactly those $\mu \in K$ which have a minimal polynomial over $\Q$ with integer coefficients and unit constant term. We are interested in so-called $c$-hyperbolic elements of $U_K$:

\begin{Def}
An algebraic unit $\mu \in U_K$ is called $c$-hyperbolic if $$\forall k \in \{1, \ldots, c\}, \hspace{1mm} \forall i_1, i_2, \ldots, i_k \in \{1,\ldots,n\}: \vert \sigma_{i_1} (\mu) \ldots \sigma_{i_k}(\mu) \vert \ne 1.$$
\end{Def}

The notion of a $c$-hyperbolic unit is the translation of $c$-hyperbolic integer-like matrices to number fields. By definition an element of $K$ is $c$-hyperbolic if and only if the companion matrix of its minimal polynomial over $\Q$ is $c$-hyperbolic and integer-like. For any $x \in U_K$, we have that the product of all $\sigma_i(x)$ is up to sign equal to the constant term of the minimal polynomial of $x$ over $\Q$ and therefore $U_K$ cannot have $n$-hyperbolic elements. In the case where $K$ is totally imaginary one can conclude in the same way that there are no units which are $\frac{n}{2}$-hyperbolic. 

\smallskip
The structure of the group $U_K$ is well known because of Dirichlet's Units Theorem. We will use the techniques of this theorem to investigate the existence of $c$-hyperbolic units in different number fields. Just as before, we distinguish between two possible cases, depending on wether or not our number field is totally imaginary.

\begin{Prop}\label{hyper}
Let $K$ be a number field of degree $n$ which is not totally imaginary. Then there exists a $c$-hyperbolic $\mu \in U_K$ for all $c \leq n-1$.
\end{Prop}

The proof is analogous to that in \cite{dv09-1} in the case of a totally real Galois extensions over $\Q$. 

\begin{proof} It is of course sufficient to prove this for $c = n-1$. Write $n = s + 2t$ as before with $s \neq 0$ by the conditions of the theorem. Assume that the first $s$ monomorphisms $\sigma_1, \ldots, \sigma_s$ are the real ones. Look at the map 
\begin{eqnarray*}
l: U_K &\to& \R^{s+t}: \\
\mu &\mapsto& \left( \log\vert \sigma_1(\mu)\vert, \ldots, \log\vert \sigma_s(\mu)\vert, \log \vert \sigma_{s+1}(\mu)\vert, \ldots, \log\vert \sigma_{s+t}(\mu) \vert \right),
\end{eqnarray*}
which is also used in the proof of Dirichlet's Units Theorem (see \cite{st87-1}). Now $l(U_K)$ is a lattice of dimension $s+t-1$ in $\R^{s+t}$, which spans the vector space $$V = \{ (x_1, \ldots, x_{s+t}) \in \R^{s+t} \mid \sum_{i=1}^s x_i +  \sum_{j = s+1}^{t} 2 x_j = 0 \}.$$ We claim that there exists an element $x \in l(U_K)$ for which
$$  \forall k \in \{1, \ldots, c\}, \hspace{1mm} \forall i_1, \ldots, i_k \in \{1, \ldots, s+t \}:  \hspace{2mm} \sum_{j=1}^k x_{i_j} \neq 0.$$
Any element $\mu \in U_K$ with $l(\mu) = x$ is then obviously a $c$-hyperbolic element. 

\smallskip

It's easy to see that each of the equations $\displaystyle \sum_{j=1}^k x_{i_j} = 0$ is linearly independent of the equation determining $V$. For this to be true, we indeed use the assumption that $c \leq n-1 = s + 2t - 1$ and $s > 0$. So the set of solutions of $\displaystyle \sum_{j=1}^k x_{i_j} = 0$ in $V$ determines a subspace of $V$ of dimension $s+t-2$. 
It is not so difficult to see that a cocompact lattice of a vector space can never be contained in a finite union of proper subspaces. So this means that we can always find an $x$ as described above.
\end{proof}

\begin{Prop}\label{hyper2}
Let $K$ be a number field of degree $n$ which is totally imaginary. Then there exists a $c$-hyperbolic $\mu \in U_K$ for all $c \leq \frac{n}{2}-1$.
\end{Prop} 

The proof is completely the same as in the previous case and thus is left for the reader. As stated above, the bounds given in the propositions are optimal.

\begin{Rmk}
From the proof of Proposition~\ref{hyper} it follows that if $K$ is a number field of degree $n$ which is not totally imaginary, then 
we can always find a $c$-hyperbolic algebraic unit $\mu$ (for any $c< n$) such that $\vert \mu \vert > 1$, 
but all other conjugates have absolute value strictly smaller than 1. So we can assume that $\mu$ is a so called Pisot number.
\end{Rmk}

\subsection{Existence of Galois extensions}

In the previous part we investigated the existence of $c$-hyperbolic units in number fields. In the proof of the main theorem, it will be necessary to take fields extensions to ensure the existence of such units for the given $c$. The following theorem shows us that we can always find a Galois extension for any given degree:

\begin{Thm} \label{ext}
Let $\Q \subseteq F$ be a finite degree field extension and $m \in \N_0$ a natural number. Then there exists a Galois extension $F \subseteq E$ such that $[E:F] = m$. If $F$ is real, then we can find such an $E$ which is not totally imaginary.
\end{Thm}

\begin{proof}
Let $n$ be the degree of the field extension $\Q \subseteq F$. From \cite[Theorem 2.3]{dv09-1} we know that there exists a totally real field extension $\Q \subseteq K$ for every given degree such that $K$ is Galois over $\Q$ and $\Gal(K,\Q)$ is abelian. In fact, the theorem doesn't say anything about the Galois group, but it follows from the proof that we can always assume this.

Let $K$ be such a field extension with $[K:\Q] = mn$. Look at the field $E = KF$, the smallest field that contains both $K$ and $F$. Since $K$ is a splitting field over $\Q$ for a polynomial $f \in \Q[X]$, $E$ is also a splitting field for the same polynomial $f$ over the field $F$. Thus $E$ is a Galois extension of $F$ and $m \mid [E:F]$.

\smallskip
We now claim that $\Gal(E,F)$ is abelian. 
Look at the homomorphism $$\pi: \Gal(E,F) \to \Gal(K,\Q): \sigma \mapsto \sigma|_K,$$
then it is easy to see that $\pi$ is injective. Since $\Gal(K,\Q)$ is abelian, therefore also $\Gal(E,F)$ must be abelian. 

\smallskip

Since $\Gal(E,F)$ is a finite abelian group and $m \mid \vert \Gal(E,F) \vert$, we can always find a (normal) subgroup $H \le \Gal(E,F)$ of index $m$. By the fundamental theorem of Galois theory, there exists a subfield $E_0$ of $E$, Galois over $F$, such that $[E_0:F] = m$ and thus the field $E$ satisfies the conditions of the theorem. If $F$ is real, then since $K$ is also real, we have that $E$ and so also $E_0$ is realized as a subfield of $\R$.
\end{proof}

If $F$ cannot be embedded in $\R$, then of course any field extension cannot be imbedded in $\R$. So the condition in this theorem is necessary. 

\subsection{Proof of the other direction of the main theorem}
For constructing matrices that commute with a representation, the following trivial observation is useful:
\begin{Lem}\label{com}
Let $M,C_1,\ldots,C_m$ be $k\times k$ matrices over any field $F$ and assume that $C_i$ commutes with $M$ for all $i$. Then the $km\times km$ matrices
\begin{align*}
A = \begin{pmatrix}
0 & 0 & \ldots & 0 & C_1 \\
\I_{k} & 0 & \ldots & 0 & C_2 \\
0 & \I_{k} & \ldots & 0 & C_3 \\
\vdots & \vdots & \ddots & \vdots & \vdots \\
0 & 0 & \ldots & \I_{k} & C_{m}
\end{pmatrix}, \hspace{2 mm} B = \begin{pmatrix} M & 0 & \hdots & 0\\ 0 & M & \hdots & 0\\ \vdots & \vdots & \ddots & \vdots \\ 0 & 0 & \hdots & M\end{pmatrix} 
\end{align*}
commute.
\end{Lem}

We are now ready to prove the main result of this paper.

\begin{proof}[Proof of Theorem~\ref{vert}]

At the end of section~\ref{firstdirection}, we already showed how to prove the first direction of this theorem.

\medskip

Now let $\rho$ be a $\Q$-irreducible representation that splits in more than $\frac{c}{m}$ components over $\R$. We construct a $c$-hyperbolic, integer-like matrix $C \in GL(\Q)$ which commutes with $m \rho = \rho + \ldots + \rho$. 
Let $\chi$ be the character of an irreducible component of $\rho$. Take a field extension $\Q \subseteq F$ of minimal degree $n$ such that $\chi$ is afforded by a $F$-representation $\rho_0: G \to \GL_k(F)$. Note that $n$ is also the number of components of $\rho$ over $\C$. We first look for a $c$-hyperbolic unit $\mu$ in some field extension of $F$. Let $E$ be a field extension like in Theorem~\ref{ext} of degree $m$. 

\smallskip

First assume that $\chi$ is not afforded by a real representation, so $F$, and hence also $E$, is totally imaginary. Since $n$ is also the number of components of $\rho$ over $\C$, we find that $n > 2 \frac{c}{m}$ and thus $[E:\Q]> 2c$ ($\Rightarrow [E:\Q]\geq 2c +2$). It follows from Proposition~\ref{hyper2} that there exits a $c$-hyperbolic unit $\mu$ in $E$. 

In the other case, $\chi$ is afforded by a real representation. Because of Theorem~\ref{nietsterk}, we can assume that $F$ is real and we can take $E$ to be not totally imaginary. This time we have that $mn > c$ and thus there exists a $c$-hyperbolic unit $\mu$ in $E$ as well. So in both cases we find a Galois extension $E$ of $F$ of degree $m$ and a $c$-hyperbolic unit in $E$.

\smallskip

Let $\sigma_1, \ldots, \sigma_n$ be the $n$ distinct monomorphisms of $F \to \C$ and take a matrix $P$ as in Theorem \ref{sigmas}, with corresponding representation $\tilde{\rho} = P^{-1} \rho P$ just as in the proof. Look at the matrix \[R= \I_m\otimes P=
\left(\begin{array}{cccc} P & 0 & \cdots & 0 \\
0 & P & \cdots & 0 \\
\vdots & \vdots & \ddots & \vdots \\
0 & 0 & \cdots & P\end{array}\right),\] then it holds that \[ R ( m\tilde{\rho} ) R^{-1} = \left(\begin{array}{cccc} P\tilde{\rho}P^{-1} & 0 & \cdots & 0 \\
0 &P \tilde{\rho} P^{-1} & \cdots & 0 \\
\vdots & \vdots & \ddots & \vdots \\
0 & 0 & \cdots &P\tilde{\rho} P^{-1}\end{array}\right) = m\rho.
\]

\smallskip

Let $$f_0(X) = \prod_{\sigma \in \Gal(E,F)}(X - \sigma(\mu)) = \sum_{j=0}^m b_j X^j \in F[X].$$ 
Take the polynomials $f_i = \sigma_i(f) = \displaystyle \sum_{j=0}^m a_{ij} X^j$ and form the matrices $$C_j = \begin{pmatrix}
a_{1j} & 0 & \ldots & 0\\
0 & a_{2j} & \ldots & 0\\
\vdots & \vdots & \ddots & \vdots \\
0 & 0 & \hdots & a_{nj}
\end{pmatrix} \otimes \I_k = \begin{pmatrix}
\sigma_1(b_j) & 0 & \ldots & 0\\
0 & \sigma_2(b_j) & \ldots & 0\\
\vdots & \vdots & \ddots & \vdots \\
0 & 0 & \hdots & \sigma_n(b_j) 
\end{pmatrix}\otimes \I_k .$$
It's easy to see that every matrix $C_j$ commutes with the representation $\tilde{\rho}$ and that $P C_j P^{-1} \in \GL_{kn}(\Q)$ because of Proposition~\ref{q}.
Now construct the matrix $$\tilde{C} = \begin{pmatrix}
0 & 0 & \ldots & 0 & -C_0 \\
\I_{kn} & 0 & \ldots & 0 & -C_1 \\
0 & \I_{kn} & \ldots & 0 & - C_2 \\
\vdots & \vdots & \ddots & \vdots & \vdots \\
0 & 0 & \ldots & \I_{kn} & - C_{m-1}
\end{pmatrix},
$$ then it is obvious that $\tilde{C}$ commutes with $m \tilde{\rho}$ because of Lemma \ref{com}. A direct computation shows that 
$$R \tilde{C} R^{-1} =  \begin{pmatrix}
0 & 0 & \ldots & 0 & -P C_0 P^{-1} \\
P P^{-1} & 0 & \ldots & 0 & -P C_1 P^{-1}  \\
0 & P P^{-1} & \ldots & 0 & -P C_2 P^{-1}  \\
\vdots & \vdots & \ddots & \vdots & \vdots \\
0 & 0 & \ldots & P P^{-1} & -P C_{m-1} P^{-1} \end{pmatrix}
$$  $$= \begin{pmatrix}
0 & 0 & \ldots & 0 & -P C_0 P^{-1} \\
\I_{kn} & 0 & \ldots & 0 & -P C_1 P^{-1}  \\
0 & \I_{kn} & \ldots & 0 & -P C_2 P^{-1}  \\
\vdots & \vdots & \ddots & \vdots & \vdots \\
0 & 0 & \ldots & \I_{kn} & -P C_{m-1} P^{-1} \end{pmatrix}$$ and thus $R \tilde{C} R^{-1} \in \GL_{knm}(\Q)$. Note that the characteristic polynomial $f(X)$ of $\tilde{C}$ equals $\left(\prod_{i=1}^n f_i(X) \right) ^k$. It's easy to check that the polynomial $f$ has coefficients in $\Q$ and all of its roots are conjugates of $\mu$. This means that $f$ is some power of the minimal polynomial of $\mu$. Therefore $\tilde{C}$ is $c$-hyperbolic and integer-like because of our choice of $\mu$. 
Since $C$ satisfies all conditions of the theorem, this completes the other direction of the main theorem.
\end{proof}

\section{Generalization to other classes of infra--nilmanifolds}
Up till now, we discussed infra--nilmanifolds modeled on a free $c$--step nilpotent Lie group, but in fact 
the results of this paper do generalize quite immediately to other classes of infra--nilmanifolds.

\smallskip

For example we can consider the Lie algebra $\lie_{c,d,r}$ which is the free $c$--step nilpotent and $d$--step solvable Lie algebra on $r$ generators 
(over $\R$) and let $G_{c,d,r}$ be the corresponding simply connected Lie group. Theorem B.\ of \cite{dv08-1} (which was slightly reformulated in \cite[Theorem 2.1]{dv11-1}) can now also be stated for manifolds 
modeled on $G_{c,d,r}$:
\begin{Thm}
Let $M$ be an infra-nilmanifold modeled on $G_{c,d,r}$, with holonomy group $H$ and associated abelianized rational holonomy 
representation $\overline{\varphi}: H \to \GL_r(\Q) $. Then the following are equivalent: 
\begin{center}
$M$ admits an Anosov diffeomorphism. \\
$\Updownarrow$ \\
There exists an integer-like $c$-hyperbolic matrix $C\in \GL_r(\Q) $ that commutes with every element of $\overline{\varphi}(H)$.
\end{center}
\end{Thm}
To prove this theorem one can follow almost word by word the original proof in \cite{dv08-1}. 

\smallskip

Having obtained this theorem, we now also get the following generalization of our main result for free:
\begin{Thm}
Let $M$ be an infra-nilmanifold modeled on $G_{c,d,r}$, with holonomy group $H$ and associated abelianized rational holonomy 
representation $\overline{\varphi}: H \to \GL_r(\Q)$. Then the following are equivalent: 
\begin{center}
$M$ admits an Anosov diffeomorphism. \\
$\Updownarrow$ \\
Every $\Q$-irreducible component of
$\bar{\varphi}$ that occurs with multiplicity $m$,\\ splits in more
than $\frac{c}{m}$ components when seen as a representation over
$\R$.
\end{center}
\end{Thm}

\section{Some applications of Theorem~\ref{main}}
\label{applic}
In \cite{dv11-1} it was already shown how Theorem~\ref{main} could be used to construct infra-nilmanifolds with an abelian holonomy group and allowing an Anosov diffeomorphism. In this section, we will now show how we can also apply this theorem in the case of non-abelian holonomy groups. We first recall some of the facts which were developed in \cite{dv11-1}.

\smallskip

Let $N$ be a torsion--free, finitely generated nilpotent group $N$, then we define for all positive integers $i$ the subgroup
$$\Gamma_i(N)=\sqrt{\gamma_i(N)}=\{x\in N\mid \exists n\in\N_0:
x^n\in\gamma_i(N)\}=N \cap \gamma_i(N_\Q) ,$$
where the $\gamma_i(N)$ indicate the terms of the lower central series of $N$. 
These groups $\Gamma_i(N)$ are fully characteristic subgroups of $N$. 

\medskip

The following theorem which was proved in \cite[Theorem 4.1]{dv11-1} is very useful to find examples of 
infra-nilmanifolds with a specific rational holonomy representation. 

\begin{Thm}\label{construct1}
Let $\varphi:H\rightarrow\Aut(N)$ be a faithful representation of a
finite group $H$ into the group of automorphisms of a
torsion-free, finitely generated nilpotent group $N$, and denote
with
$$\bar{\varphi}_i:H\to\Aut\left(\frac{\Gamma_i(N)}{\Gamma_{i+1}(N)}\right)\cong\Aut(\Z^{k_i})$$ the induced 
morphism. If there exists, for some positive integer $i$, a torsion-free extension
$$1\to\frac{\Gamma_i(N)}{\Gamma_{i+1}(N)}\to\bar{\Gamma}\to H\to 1$$ inducing
$\bar{\varphi}_i$, then there exists an almost--Bieberbach group
$\Gamma$ with holonomy group $H$, whose translation subgroup is a
finite index subgroup of $N$ and such that the rational holonomy
representation $\psi:~H\to\Aut(N_\Q)$ coincides with
$\varphi:H\to\Aut(N)\subseteq\Aut(N_\Q)$.
\end{Thm}

Now, we let $N_{r,c}$ be the free $c$--step nilpotent group on $r$ generators and we use 
$N_{r,c,\Q}$ to denote the rational Mal'cev completion of $N_{r,c}$. The corresponding Lie group (so the real Mal'cev completion) is the free $c$--step nilpotent Lie group 
on $r$ generators. We have the natural homomorphisms
\[ \mu: \Aut(  N_{r,c} ) \rightarrow \Aut( N_{r,c}/[N_{r,c},N_{r,c}])\cong \GL_r(\Z)\mbox{ and }\]
\[\mu_\Q: \Aut(  N_{r,c,\Q} ) \rightarrow \Aut( N_{r,c,\Q}/[N_{r,c,\Q},N_{r,c,\Q}])\cong \GL_r(\Q),\]
which are both onto. For making the identifications with $\GL_r(\Z)$ and $\GL_r(\Q)$ we fix a 
set of generators $x_1,x_2,\ldots,x_r$ of $N_{r,c}$ and the images of these these generators in the respective abelianizations 
give rise to a basis of the free abelian group $N_{r,c}/[N_{r,c},N_{r,c}]$ and the $\Q$--vector space $N_{r,c,\Q}/[N_{r,c,\Q},N_{r,c,\Q}]$ w.r.t.\ which we 
represent an automorphism as a matrix.\\
We will use the following observation: 
\begin{Lem}\label{lift finite group}
Let $F\subseteq \GL_r(\Z)$ be a finite subgroup. Then there exist a finite subgroup $\tilde{F}$ of
 $ \Aut(  N_{r,c,\Q} )$ and a finitely generated subgroup $N\subseteq N_{r,c,\Q} $ such that 
\begin{enumerate}
\item $\mu_\Q(\tilde{F})=F$,
\item $N$ contains $N_{r,c}$ as a subgroup of finite index,
\item $N\gamma_2(N_{r,c,\Q})= N_{r,c} \gamma_2 (N_{r,c,\Q})$ and 
\item $\forall \alpha \in \tilde{F}: \; \alpha(N) =N$.
\end{enumerate}
\end{Lem}
\begin{proof}
The existence of a finite subgroup $\tilde{F}\subseteq \Aut(  N_{r,c,\Q} )$ such that $\mu_\Q(\tilde{F})=F$ is a result 
which is due to Kuz'min (see \cite[page 91]{kuzm99-1}). Now, consider $N_{r,c,\Q}\rtimes \tilde{F}$ and let $\tilde{N}$ be the 
subgroup of $N_{r,c,\Q}\rtimes \tilde{F}$ which is generated by $N_{r,c}$ and $\tilde{F}$. As both $N_{r,c}$ and 
$\tilde{F}$ are finitely 
generated, we have that $\tilde{N}$ is finitely generated and hence also $N=\tilde{N} \cap N_{r,c,\Q}$, which is of 
finite index in $\tilde{N}$, is finitely generated. By this construction $N$ satisfies properties  2., 3.\ and 4.\ in the statement 
of this lemma.
\end{proof}
We can now prove the following:
\begin{Thm}\label{examp1}
Let $H$ be any finite group and $c$ be any positive integer. Then there exists a positive integer $K$ such that for any 
$k\geq K$ there is a infra-nilmanifold which is modeled on the free $c$--step nilpotent Lie group on $k$ generators, admits an Anosov diffeomorphism and has $H$ as its holonomy group.
\end{Thm}
\begin{proof}
It is well known that any finite group $H$ can be realized as the holonomy group of a flat manifold (\cite{ak57-1}). Hence, there exists a representation 
\[ \psi:H\rightarrow \GL_n(\Z)\]
and a torsion-free extension 
\begin{equation}\label{extensie}
0 \to \Z^n \to \bar{\Gamma}_1 \to H\to 1
\end{equation}
inducing $\psi$. Now, take $K=(c+1)(n+1)$ and choose any $k\geq K$. Let 
$\varphi_1: H \rightarrow \GL_k(\Z)$ be the representation 
\[ \varphi_1= \underbrace{ \psi\oplus \psi\oplus \cdots \oplus \psi}_{c+1\mbox{ times }} \oplus \underbrace{
1 \oplus 1\oplus \cdots \oplus 1}_{k - (c+1) n  \mbox{ times}} \]
where $1$ denotes the trivial $1$--dimensional representation. Note that there are at least $c+1$ of these trivial factors.  Let $f$ be a $2$-cocycle such that 
the cohomology class   $\langle f \rangle\in H^2(H,\Z^n)$ describes the extension (\ref{extensie}), where of course
$\Z^n$ is a $H$--module via $\psi$. We can decompose the $H$--module $\Z^k$ (via $\varphi_1$) as 
a direct sum $\Z^n \oplus \Z^{k-n}$ where $H$ acts on the $\Z^n$--part via $\psi$ and on the $\Z^{k-n}$--part 
via $c$ times $\psi$ and $k-(c+1)n$ times the trivial representation. Then $H^2(H,\Z^k) = H^2(H,\Z^n) \oplus 
H^2(H,\Z^{k-n})$ and the $2$--cohomology class corresponding to $\langle f \rangle \oplus \langle 0 \rangle$ 
determines an extension 
\[  0 \to \Z^k \to \bar{\Gamma} \to H \to 1\]
which is also torsion-free and which induces $\varphi_1: H \rightarrow \GL_k(\Z)$.\\
Now, let $N_{k,c}$ be the free $c$--step nilpotent group on $k$ generators. By Lemma \ref{lift finite group}, we can find 
a group $N$ containing $N_{k,c}$ as a subgroup of finite index and a morphism $\varphi:H \to  \Aut(N)$ which induces $\varphi_1$ on 
\[ \frac{\Gamma_1(N)}{\Gamma_2(N)}= \frac{N}{N\cap \gamma_2(N_{k,c,\Q})}=
\frac{N\gamma_2(N_{k,c,\Q})}{\gamma_2(N_{k,c,\Q})}=
\frac{\Gamma_1(N_{k,c})}{\Gamma_2(N_{k,c})}\cong \Z^k.\]

By applying Theorem~\ref{construct1} we can conclude that there exists an almost--Bieberbach group $\Gamma$ 
with holonomy group $H$ and whose translation subgroup is of finite index in $N$. Hence, $\Gamma$ determines an infra-nilmanifold $M$ which is modeled on the free $c$--step nilpotent Lie group on $k$ generators and 
has $H$ as its holonomy group. Moreover, the rational holonomy representation of $M$ coincides with 
$\varphi:H \rightarrow\Aut(N) \subseteq \Aut(N_{k,c,\Q})$ and it is obvious that the abelianized holonomy 
representation, then coincides with $\varphi_1:H \rightarrow \GL_k(\Z) \subseteq \GL_k(\Q)$. By construction,
each $\Q$--irreducible component of $\varphi_1$ occurs with at least multiplicity $c+1$, and hence the conditions of Theorem~\ref{main} are trivially satisfied, which allows us to conclude that $M$ admits an Anosov diffeomorphism
\end{proof}

The previous result shows that any finite group can appear as the holonomy group of an infra-nilmanifold with an Anosov diffeomorphism and which is modeled on a free $c$--step nilpotent Lie group, but we have to allow large enough dimensions. For the next result, we will fix the smallest non-abelian group and determine such an infra-nilmanifold of minimal dimension. 

\medskip

So let $H=D_3=\langle a,b\;| \;a^3 =b^2=  (ab)^2=1\rangle$ be the dihedral group of order six (or the symmetric group of degree 3). The conjugacy classes of $D_3$ are $\{1\}$, $\{a, a^2\}$ and $\{b,ab,a^2b\}$. The character table of 
$D_3$ is given by 
\[ 
\begin{array}{c|ccc}
 & 1 & a & b \\ \hline
 \chi_1 & 1 & 1 & 1 \\
 \chi_2 & 1 & 1 & -1\\
 \chi_3 & 2 & -1 & 0
 \end{array}.
\]
These characters correspond to the representations 
\[ \rho_1: D_3 \rightarrow \GL_1(\Q) \mbox{ with } \rho_1(a)=1 \mbox{ and } \rho_1(b)=1;\]
\[ \rho_2: D_3 \rightarrow \GL_1(\Q) \mbox{ with } \rho_1(a)=1 \mbox{ and } \rho_1(b)=-1;\]
\[ \rho_3: D_3 \rightarrow \GL_2(\Q) \mbox{ with } \rho_1(a)=\begin{pmatrix}0 & -1\\1 & -1\end{pmatrix}
\mbox{ and } \rho_1(b)=\begin{pmatrix} 0 & -1 \\ -1 & 0 \end{pmatrix}.\]
It follows that the $\Q$--irreducible representations of $H$ are also $\C$--irreducible. Note that all representations 
are already given as integral representations. We remark that the representation 
\[ \rho'_3:H \rightarrow \GL_2(\Q) \mbox{ with } \rho_1(a)=\begin{pmatrix}0 & -1\\1 & -1\end{pmatrix} 
\mbox{ and } \rho_1(b)=\begin{pmatrix} 0 & 1 \\ 1 & 0 \end{pmatrix}\]
is equivalent to $\rho_3$ when considered as a representation over $\Q$, but they are not when considered over $\Z$.
Hence, there are 4 non-equivalent irreducible $\Z$--representations of $H$, see e.g.\ \cite{bbnwz78-1}.

\medskip

Now assume that $M=\Gamma\backslash G$ is an infra-nilmanifold which is modeled on a free $c$--step nilpotent 
Lie group, which has $H$ as its holonomy group and such that $M$ admits an Anosov diffeomorphism. The abelianized rational holonomy representation $\varphi: D_3 \rightarrow \GL_n(\Q)$ of $M$ must be faithful and hence it must contain at least one component which is $\Q$--equivalent to $\rho_3$. As this component is $\R$--irreducible and $M$ admits an Anosov diffeomorphism, it must in fact appear at least $c+1$ times by Theorem~\ref{main}. We will now show that this 
lower bound is sharp.
\begin{Prop}\label{D3}
Let $c>1$. Then there exists an infra-nilmanifold $M$ which is modeled on a free $c$--step nilpotent Lie group and such that $M$ has holonomy group $D_3$, admits an Anosov diffeomorphism and its abelianized 
holonomy representation is equivalent to $\underbrace{\rho_3\oplus \rho_3 \oplus \cdots \oplus \rho_3}_{c+1\mbox{ times}}$.
\end{Prop}

Before we can give the proof of this proposition, we need to recall one more result from \cite{dv11-1} on totally reducible integral representations. A representation $\rho:F \rightarrow \GL_n(\Z)$ is said to be totally 
reducible if and only if $\Z^n$ splits as a direct sum of $\Z$--irreducible submodules.
The following lemma can be found in \cite[Lemma 4.3]{dv11-1}:
\begin{Lem}\label{M}
Let $N$ be a finitely generated torsion-free  nilpotent
group and $\varphi:F\to\Aut(N_\Q)$ a morphism, with $F$ a finite
group, such that $\varphi(f)(N)=N$ for all $f\in F$. Then there
exists a finitely generated subgroup $N^\prime$ of $N_\Q$ such that
\begin{itemize}
\item $\varphi(f)(N^\prime)=N^\prime$ for all $f\in F$, \item $N$ is a subgroup
of $N^\prime$ of finite index, and \item for all positive integers $i$, the induced
representation
$$\bar{\varphi}_i:F\to\Aut\left(\frac{\Gamma_i(N^\prime)}{\Gamma_{i+1}(N^\prime)}\right)=\GL(n_i,\Z)$$ is totally reducible.
\end{itemize}
\end{Lem}

\begin{proof}[Proof of  Proposition~\ref{D3}] 
Let $N_{2(c+1),c}$ be the free nilpotent group on $2(c+1)$ generators and of class $c$. 
Let us denote the generators of $N_{2(c+1),c}$ by $x_1, x_2, \ldots ,x_{2(c+1)}$. By Lemma~\ref{lift finite group} there exists a finitely generated subgroup $N$
of $N_{2(c+1),c,\Q}$ containing $N_{2(c+1),c}$ as a subgroup of finite index and a morphism $\varphi:D_3\to \Aut(N)$ 
such that the  induced morphism $\varphi_1: D_3 \rightarrow \Aut(\Gamma_1(N)/\Gamma_2(N)) = \GL_{2(c+1)}(\Z)$ is given by 
\[ \varphi_1 = \underbrace{\rho_3 \oplus \rho_3 \oplus \cdots \oplus  \cdots \oplus \rho_3}_{c+1\mbox{ times}}.\]
Now let $y_{i,j}=[x_i,x_j] $,  $(1\leq i<  j \leq 2(c+1))$. The natural projections $\bar{y}_{i,j}$ of these elements in 
$\Gamma_2(N)/\Gamma_3(N)$ form a basis of a free abelian group which is of finite index in this quotient. We claim that the subgroup, say $L$,
generated by 
$\bar{y}_{1,3}, \bar{y}_{1,4}, \bar{y}_{2,3}, \bar{y}_{2,4}$ is a $D_3$ submodule for the action of $D_3$ 
given by $\varphi_2: D_3 \rightarrow \Aut(\Gamma_2(N)/\Gamma_3(N))$.  In fact, we can compute this action explicitly.
Note that by the form of $\varphi_1$, we have that 

\[ \begin{array}{ll}
\varphi(a) (x_1)= x_2\mod \Gamma_2(N) \hspace{4mm}&  \varphi(b) (x_1) = x_2^{-1} \mod \Gamma_2(N) \\
\varphi(a) (x_2)= x_1^{-1}x_2^{-1}\mod \Gamma_2(N) \hspace{4mm}&  \varphi(b) (x_2) = x_1^{-1} \mod \Gamma_2(N) \\
\varphi(a) (x_3)= x_4\mod \Gamma_2(N) \hspace{4mm}&  \varphi(b) (x_3) = x_4^{-1} \mod \Gamma_2(N) \\
\varphi(a) (x_4)= x_3^{-1}x_4^{-1}\mod \Gamma_2(N) \hspace{4mm}&  \varphi(b) (x_4) = x_3^{-1} \mod \Gamma_2(N) 
\end{array}\]

From these identities we can compute the action of $a$ and $b$ on each $y_{i,j}$ $\bmod \hspace{1mm} \Gamma_3(N)$, e.g.
\begin{eqnarray*}
\varphi(a)(y_{1,4} ) & = & [\varphi(a) (x_1), \varphi(a)(x_4) ] \bmod \Gamma_3(N)\\ 
                     & = & [x_2, x_3^{-1}x_4^{-1} ] \bmod \Gamma_3(N)\\ 
                     & = & [x_2, x_3]^{-1}[x_2,x_4]^{-1} \bmod \Gamma_3(N)\\ 
                     & = & y_{2,3}^{-1}y_{2,4}^{-1}  \bmod \Gamma_3(N).
\end{eqnarray*}
When doing this for all generators, we find that w.r.t.\ the basis $ \bar{y}_{1,3}, \bar{y}_{1,4}, \bar{y}_{2,3}, \bar{y}_{2,4}$  of
$L$, the action of $a$ and $b$ are represented by matrices 
\[ A= \left( \begin{array}{cccccc}
 0 & 0 & 0 & 1 \\
 0 & 0 &-1 & 1 \\
 0 &-1 & 0 & 1 \\
 1 &-1 &-1 & 1 
\end{array}\right) \mbox{ and }
B= \left( \begin{array}{cccccc}
 0 & 0 & 0 & 1 \\
 0 & 0 & 1 & 0 \\
 0 & 1 & 0 & 0 \\
 1 & 0 & 0 & 0 
\end{array}\right) .
 \]
The character $\chi$ of this representation satisfies $\chi(1)=4$, $\chi(a)= 1$ and $\chi(b)=0$ and so $\chi=\chi_1 + \chi_2 + \chi_3$, from which it follows
that the representation determined by the matrices $A$ and $B$ is $\Q$--equivalent to $\rho_1\oplus  \rho_2 \oplus \rho_3$. Now, we use Lemma~\ref{M} to 
find a subgroup $N^\prime$ of $N_{2(c+1),c,\Q}$ which contains $N$ as a finite index subgroup, such that $\varphi$ can also be seen as a representation $\varphi:D_3 \to \Aut(N')$ and 
such that for each $i$ 
\[ \varphi_i: D_3 \to \Aut(\Gamma_i(N')/\Gamma_{i+1}(N'))\]
is totally reducible. By the above, we know that $\varphi_2$ must have a subrepresentation which is $\Q$--equivalent to $\rho_1\oplus \rho_2 \oplus \rho_3$.
Hence as a representation over $\Z$ it must contain a subrepresentation which is $\Z$-equivalent to $\rho=\rho_1\oplus \rho_2\oplus \rho_3$ or to 
 $\rho'=\rho_1\oplus \rho_2\oplus \rho_3'$. Both of these representations occur as the holonomy representation of a 4-dimensional Bieberbach group (\cite{bbnwz78-1}).
Hence, by an analogous argument as in the proof of Theorem~\ref{examp1}, we can conclude that there 
is a torsion-free extension
\[ 1 \to \frac{\Gamma_2(N')}{\Gamma_3(N')} \to \bar\Gamma \to D_3 \to 1\]
inducing the representation $\varphi_2$.

It follows from Theorem~\ref{construct1} that there exists an almost--Bieberbach group $\Gamma$ with holonomy group $D_3$ and such that 
the corresponding rational holonomy representation  $\psi:D_3\rightarrow \Aut(N_{2(c+1),c,\Q})$ coincides with $\varphi$. This implies that the abelianized rational holonomy representation is $\Q$--equivalent to $(c+1) \rho_3$, hence by our main Theorem~\ref{main}, the infra-nilmanifold 
determined by $\Gamma$ admits an Anosov diffeomorphism, which finishes the proof. 

\end{proof}

\begin{center}
\textbf{Acknowledgments}
\end{center}
We would like to thank the referee for his/her useful comments which have led to the inclusion of Section~\ref{applic}.


\begin{thebibliography}{10}

\bibitem{ak57-1}
Auslander, L. and Kuranishi, M.
\newblock {\em On the holonomy group of locally Euclidean spaces}.
\newblock Ann. of Math. (2), 1957, 65, pp. 411--415.

\bibitem{bbnwz78-1}
Brown, H., {B\"ulow}, R., Neu{b\"u}ser, J., Wondratscheck, H., and Zassenhaus,
  H.
\newblock {\em Crystallographic groups of four--dimensional Space}.
\newblock Wiley New York, 1978.

\bibitem{dm05-1}
Dani, S.~G. and Mainkar, M.~G.
\newblock {\em Anosov automorphisms on compact nilmanifolds associated with
  graphs}.
\newblock Trans. Amer. Math. Soc., 2005, 357 6, 2235--2251.

\bibitem{deki99-1}
Dekimpe, K.
\newblock {\em Hyperbolic automorphisms and {A}nosov diffeomorphisms on
  nilmanifolds}.
\newblock Trans. Amer. Math. Soc., 2001, 353 7, pp.\ 2859--2877.

\bibitem{deki11-1}
Dekimpe, K.
\newblock {\em What an infra-nilmanifold endomorphism really should be \ldots}.
\newblock Topol. Methods Nonlinear Anal., 2012, 40 1, 111--136.

\bibitem{dd03-2}
Dekimpe, K. and Deschamps, S.
\newblock {\em Anosov diffeomorphisms on a class of 2-step nilmanifolds}.
\newblock Glasgow.\ Math.\ J., 2003, 45 pp. 2269--280.

\bibitem{dv08-1}
Dekimpe, K. and Verheyen, K.
\newblock {\em Anosov diffeomorphisms on infra-nilmanifolds modeled on a free
  2-step nilpotent Lie group}.
\newblock Groups, Geometry and Dynamics, 2009, 3 4, pp. 555--578.

\bibitem{dv09-1}
Dekimpe, K. and Verheyen, K.
\newblock {\em Anosov diffeomorphisms on nilmanifolds modeled on a free
  nilpotent Lie group}.
\newblock Dynamical Systems -- an international journal, 2009, 24 1, pp.
  117--121.

\bibitem{dv11-1}
Dekimpe, K. and Verheyen, K.
\newblock {\em Constructing infra-nilmanifolds admitting an {A}nosov
  diffeomorphism}.
\newblock Adv. Math., 2011, 228 6, 3300--3319.

\bibitem{fg12-1}
Farrell, F.~T. and Gogolev, A.
\newblock {\em Anosov diffeomorphisms constructed from {$\pi_k({\rm
  Diff}(S^n))$}}.
\newblock J. Topol., 2012, 5 2, 276--292.

\bibitem{hers53-1}
Herstein, I.~N.
\newblock {\em Finite multiplicative subgroups in division rings}.
\newblock Pacific J. Math., 1953, 3 121--126.

\bibitem{hj96-1}
Horn, R.~A. and Johnson, C.~R.
\newblock {\em Matrix analysis}.
\newblock Cambridge University Press, 1996.

\bibitem{isaa76-1}
Isaacs, I.~M.
\newblock {\em Character theory of finite groups}, volume~69 of {\em Pure and
  Applied Mathematics}.
\newblock New York-San Francisco-London: Academic Press, 1976.

\bibitem{kuzm99-1}
Kuz'min, Y.
\newblock Homological methods in group theory.
\newblock In {\em Summer {S}chool in {G}roup {T}heory in {B}anff, 1996},
  volume~17 of {\em CRM Proc. Lecture Notes}, pages 81--98. Amer. Math. Soc.,
  Providence, RI, 1999.

\bibitem{laur03-1}
Lauret, J.
\newblock {\em Examples of {A}nosov diffeomorphisms}.
\newblock J. Algebra, 2003, 262 1, 201--209.

\bibitem{laur08-1}
Lauret, J.
\newblock {\em Rational forms of nilpotent {L}ie algebras and {A}nosov
  diffeomorphisms}.
\newblock Monatsh. Math., 2008, 155 1, 15--30.

\bibitem{lw08-1}
Lauret, J. and Will, C.~E.
\newblock {\em On {A}nosov automorphisms of nilmanifolds}.
\newblock J. Pure Appl. Algebra, 2008, 212 7, 1747--1755.

\bibitem{lw09-1}
Lauret, J. and Will, C.~E.
\newblock {\em Nilmanifolds of dimension {$\leq 8$} admitting {A}nosov
  diffeomorphisms}.
\newblock Trans. Amer. Math. Soc., 2009, 361 5, 2377--2395.

\bibitem{main06-1}
Mainkar, M.~G.
\newblock {\em Anosov automorphisms on certain classes of nilmanifolds}.
\newblock Glasg. Math. J., 2006, 48 1, 161--170.

\bibitem{main12-1}
Mainkar, M.~G.
\newblock {\em Anosov {L}ie algebras and algebraic units in number fields}.
\newblock Monatsh. Math., 2012, 165 1, 79--90.

\bibitem{mw07-1}
Mainkar, M.~G. and Will, C.~E.
\newblock {\em Examples of {A}nosov {L}ie algebras}.
\newblock Discrete Contin. Dyn. Syst., 2007, 18 1, 39--52.

\bibitem{malf97-3}
Malfait, W.
\newblock {\em Anosov diffeomorphisms on nilmanifolds of dimension at most
  six.}
\newblock Geometriae Dedicata, 2000, 79 (3), 291--298.

\bibitem{payn09-1}
Payne, T.~L.
\newblock {\em Anosov automorphisms of nilpotent {L}ie algebras}.
\newblock J. Mod. Dyn., 2009, 3 1, 121--158.

\bibitem{port72-1}
Porteous, H.~L.
\newblock {\em Anosov diffeomorphisms of flat manifolds.}
\newblock Topology, 1972, 11, pp. 307--315.

\bibitem{serr77-1}
Serre, J.-P.
\newblock {\em Linear representations of finite groups}, volume~42 of {\em
  Graduate Texts in Math.}
\newblock New York-Heidelberg-Berling: Springer--Verlag, 1977.

\bibitem{st87-1}
Stewart, I. and Tall, D.
\newblock {\em Algebraic Number Theory, second edition}.
\newblock Chapman and Hall Mathematics Series. Chapman and Hall, London, 1987.

\end{thebibliography}
\end{document}